\documentclass[reqno]{amsart}
\usepackage{graphicx} 
\usepackage{amsfonts,amsmath,amsthm,amssymb, 
mathtools,hyperref,verbatim}
\usepackage{geometry} 
\usepackage{pgfplots}
\usepackage[backend=biber,style=numeric,sorting=nyt,giveninits=true]{biblatex}
\usepackage{hyperref}
\usepackage{url}

\pgfplotsset{width=10cm,compat=1.9}
\title{Constructions of Sequences of Alternating Sum and Difference Dominated Sets}
\author{Yorick Herrmann}
\email{\textcolor{blue}{\href{mailto:yherrman@uci.edu}{yherrman@uci.edu}}}

\author{Connor Hill}
\email{\textcolor{blue}{\href{mailto:hill.connor.03@gmail.com}{hill.connor.03@gmail.com}}}

\author{Merlin Phillips}
\email{\textcolor{blue}{\href{mailto:merlin.phillips216@gmail.com}{merlin.phillips216@gmail.com}}}
\author{Daniel Flores}
\email{\textcolor{blue}{\href{mailto:flore205@purdue.edu}{flore205@purdue.edu}}}

\author[Miller]{Steven J. Miller}
\email{\textcolor{blue}{\href{mailto:sjm1@williams.edu}{sjm1@williams.edu}}
\textcolor{blue}{\href{Steven.Miller.MC.96@aya.yale.edu}{Steven.Miller.MC.96@aya.yale.edu}}}
\address{Department of Mathematics, Williams College,
Williamstown, MA 01267, USA}

\author{Steven Senger}
\email{\textcolor{blue}{\href{mailto:StevenSenger@missouristate.edu}{StevenSenger@missouristate.edu}}}
\address{Department of Mathematics, Missouri State University,
Springfield, MO 65897, USA}

\date{\today}

\newcounter{master}[section]

\newtheorem{thm}[master]{Theorem}
\newtheorem{lem}[master]{Lemma}
\newtheorem{defn}[master]{Definition}
\newtheorem*{rem}{Remark}

\makeatletter
\let\c@table\c@master

\makeatother

\newcommand{\sss}{\big|A+A\big|}
\newcommand{\dss}{\big|A-A\big|}

\begin{document}

\maketitle

\begin{abstract}
A More Sums Than Difference (MSTD) set is a finite set of integers $A$
where the cardinality of its sumset, $A+A$, is greater than the
cardinality of its difference set, $A-A$. Since addition is commutative
while subtraction isn't, it was conjectured that MSTD sets are rare. As
Martin and O'Bryant proved a small (but positive) percentage are MSTD, it
is natural to ask what additional properties can we impose on a chain of
MSTD sets; in particular, can we construct a sequence of sets alternating
between being MSTD and More Difference Than Sums (MDTS) where each properly contains the previous? We
provide several such constructions; the first are trivial and proceed by
filling in all missing elements from the minimum to maximum elements of
$A$, while the last is a more involved construction that prohibits adding
any such elements.
\end{abstract}

\section{Introduction}
For a finite set of integers $A$, we define the sumset $A+A \coloneqq \{a+b : a,b \in A\}$. Similarly, we define the difference set $A-A \coloneqq \{a-b : a,b \in A\}$. Let $\big|A\big|$ refer to the cardinality of $A$. Intuitively, because $a+b = b+a$ for all $a,b \in \mathbb{Z}$, but $a-b \ne b-a$ unless $a=b$, we would expect $\big|A-A\big|$ to be larger than $\big|A+A\big|$ \cite{MartinOBryant, nathanson2006}. However, this is not always the case. We say $A$ is a \textit{More Sums Than Differences} (MSTD) set if $\sss \ > \ \dss$. Alternatively, $A$ is a \textit{More Differences Than Sums} (MDTS) set if $\dss \ > \ \sss$. These sets are also referred to as sum-dominated or difference-dominated sets. 

The study of MSTD sets originated with early examples discovered by Conway $\big(\{0,2,3,4,7,11,12,14\}\big)$, Marica $\big(\{1, 2, 3, 5, 8, 9, 13, 15, 16\}\big)$ \cite{MaricaConwayConjecture}, and Freiman-Pigarev
\footnote{$\{0, 1, 2, 4, 5,9, 12, 13, 14, 16, 17, 21, 24, 25, 26, 28, 29\}$}\cite{FreimanPigarev1973}.
Nathanson later formalized the problem, proved several properties of MSTD sets, provided several methods for constructing MSTD sets, and remarked that MSTD sets should be rare \cite{nathanson2006}. Surprisingly, Martin and O'Bryant \cite{MartinOBryant} proved that as $n \to \infty$, a positive percentage of subsets of $\{1,2,\dots,n\}$ are MSTD, and Hegarty proved that the set Conway found is the smallest possible MSTD set in terms of cardinality and diameter \cite{HegartyMinSum}. Additional research has focused on explicit constructions of generalized MSTD sets \cite{miller2008,ExplicitLargeFamilies,nathanson2017} and generalizations to groups \cite{Zhao2010}. Many problems in number theory involve sumsets and difference sets.

At the recent 2025 CANT (Combinatorial and Additive Number Theory Conference), Samuel Allen Alexander posed the problem of finding an infinite sequence of sets with $A_{i-1} \subset A_i$ that alternate being MSTD and MDTS. We address this problem by presenting several methods for constructing such sequences. Before introducing these methods, we must first establish several known properties of MSTD sets, and sets in general.

First, for any set of integers $A$, we define the dilation $x \cdot A \coloneqq \{xa:a \in A\}$. For all integers $x,y$ such that $x \ne  0$, the set
\begin{align}B \ = \ x \cdot A+\{y\} \ = \ \{xa+y:a\in A\}\end{align} 
satisfies $\big|B+B\big| = \sss$ and $\big|B-B\big| = \dss$ \cite{nathanson2006}. This property ensures that any method developed on a specific subset of the integers can be extended, via dilation and translation, to the full domain of integers. Furthermore, through dilation and translation, we have that any non-trivial $k$-term arithmetic progression can be obtained from the interval 
\[I \ = \ \{1,2,\ldots,k\} \ := \ [1,k].\] 
It is also evident that \begin{align}
\notag I+I \ &= \ [2,2k]\ \\ 
I-I \ &= \ [1-k,k-1],
\end{align} so such sets have the same number of sums and differences. For instance, if $A=\{1,2,3\}$, then \begin{align}
   \notag A+A \ &= \ \{2,3,4,5,6\} \\
   \notag A-A \ &= \ \{-2,-1,0,1,2\} \\
   \sss \ &= \ \dss \ = \ 5.
\end{align}

For a set of integers $A$, if there exists an $a^* \in \mathbb{Z}$ such that $a^*-A = A$, then we say that $A$ is symmetric with respect to $a^*$, and we have that $\sss = \dss$ \cite{nathanson2006}. If $A$ is an arithmetic progression, we have that $A$ is symmetric to $a^*=\min(A)+\max(A)$. For instance, the arithmetic progression $A=\{1,3,5\}$ is symmetric with respect to $6$. 

Many methods that are used to construct MSTD sets make use of the symmetry property \cite{nathanson2006}. For example, the set $A = \{0,2,3,7,11,12,14\}$ is symmetric with respect to $14$, so $\sss = \dss$. However, adding $4$ to this set forms Conway's set, which is MSTD. We also use this symmetry property in one method presented in the paper.

Section 2 of the paper presents two methods that construct the desired alternating sequence through the use of `filling in' missing elements. Theorem \ref{thm:2.4MDTSinterval} provides a method to construct an MDTS superset from any MSTD set, which forms the foundation for Filling in Method 1. Theorem \ref{THMwarmup1method2} then builds on known constructions of MSTD sets to generate sequences with controlled growth in Filling in Method 2. Section 3 presents another constraint, requiring that any element missing from an intermediate set remains missing in all subsequent sets.
Theorem \ref{NFIMethod2Theorem} provides a method based on choosing fringe elements to control the sum and difference sets.
Lastly, Section 4 summarizes the growth rate of each method.

The first constructions are trivial as we add missing elements from
the minimal to maximal element in our set and then append whatever left
and right fringe are needed. The final construction is more interesting
and involved, as we force the constructions to be non-trivial by
forbidding the inclusion of any of the missing elements. We further
develop such methods in a companion paper.

\section{Filling in Methods}
For a set $A \subset [a,b]$, `filling in' the set $A$ refers to the process of adding some (or potentially all) missing elements of $[a,b] \setminus A$ to $A$. For instance, we could fill in the set $A  =  \{1,3,5\}$ by adding either $2$, $4$, or both $2$ and $4$ to $A$. 

We provide two methods that construct the desired sequence through the use of filling in. In each method, we first fill in $A_{i-1}$ and subsequently adjoin elements to create either a MSTD or MDTS set $A_{i}$. 
\subsection{Filling in Method 1}

Before we begin, we recall the following results, taken from Nathanson \cite{nathanson2006}.

\begin{lem} \label{lemma1}
    Let $m \geq 4$ and $2\leq r  \leq m - 3$, with $m,r \in \mathbb{N}$. Let $B = [0,m-1] \setminus \{r\}$. Then $B+B  =  [0,2m-2]$, and $B-B  =  [-(m-1),m-1]$.

\end{lem}

\begin{defn}
For integers $d \geq 1$, a \emph{$d$-dimensional arithmetic progression} is a set of the form
\[
L \ := \ \{a + x_1m_1 + \dots + x_dm_d : \ell_i \leq x_i \leq \ell_i + k_i - 1 \text{ for } i = 1, \dots, d\},
\]
where $a$ is the base point,
 $m_1,\dots,m_d$ are the step sizes,
$\ell_1,\dots,\ell_d$ are the starting indices, and
$k_1,\dots,k_d$ are the lengths in each dimension.

A 0-dimensional arithmetic progression is just a single point $\{a\}$.
\end{defn}

\begin{thm}  \label{Thm1}

Let $m \ge 4$, $m \in \mathbb{N}$ and let $B$ be a subset of $[0, m-1]$ such that 
$$B+B \ = \ [0,2m-2],$$ and $$B-B \ = \ [-m+1, m+1].$$

Let $L^*$ be a $(d-1)$-dimensional arithmetic progression contained in $[0, m-1]\setminus B$ such that $\min(L^*)-1\in B$ and $m \notin L^* +L^*$. Let $k \ge 2$ and let $L$ be the d-dimensional arithmetic progression $$L \ = \ (m-L^*)+m*[1,k].$$
Let $$a^* \ = \ \min(L)+\max(L) \ = \ (k+3)m-\min(L^*)-\max(L^*),$$ and $$A^* \ = \ B\cup L\cup (\{a^*\}-B).$$ Then $A = A^*\cup \{m\}$ is a MSTD set of integers.

\end{thm}

In addition to these results, we provide a simple construction which will be used to build the MDTS sets of our sequence.

\begin{thm}\label{thm:2.4MDTSinterval}
Let $m,p \in \mathbb{N}$ such that $p > m+1$. Define 
$$
A \ \coloneqq \ [0,m] \cup \{p\}.
$$
Then $A$ is a MDTS set of integers that satisfies
$$
|A-A| - |A+A| \ =
\ \begin{cases}
m, & \text{if } p > 2m, \\
p - m - 1, & \text{if } m+1 < p \leq 2m.
\end{cases}
$$
\end{thm}

\begin{proof}

For $m,p \in \mathbb{N}$ such that $p > m+1$, define
\begin{equation} \label{eq:A2def}
A \ := \ [0, m] \cup \{p\}.
\end{equation}

We first compute the values of the sumset and difference set of $A$: \begin{align} 
\notag A + A \ &= \ \left[0, 2m\right] \cup \left[p, m+p\right] \cup \left\{2p\right\}  \\
A - A \ &= \ \left[-m, m\right] \cup \left[-p, m-p\right] \cup \left[p-m, p\right]. \label{eq:A2sumdiff}
\end{align}

To count the cardinality of each set, we consider two cases.
\begin{enumerate}
\item[] \textbf{Case 1:} $p > 2m$. In this case, since none of the subintervals overlap, counting the cardinality of each set simplifies to adding together the total number of elements in each interval:
\begin{align}
\notag \big| A + A \big| \ &= \ (2m+1) + (m+1) +1 \ = \ 3m +3 
\notag \\
\big| A - A \big| \ &= \ (2m+1) + (m+1) + (m+1) \ = \ 4m+3. \label{eq:case1sum/diff}
\end{align}

Thus, $A$ is difference dominated with exactly $m$ more differences than sums.

\item[] \textbf{Case 2:} $m+1 < p \leq 2m$. Since $[0,2m] \cup [p, m+p]  =  [0,m+p]$ and $m+p < 2p$, we have
\begin{equation} \label{eq:case2sum}
\big|A + A\big| \ = \ (m+p+1)+1 \ = \ m+p+2. 
\end{equation}
Similarly, $\left[-p, m-p\right] \cup [-m,m] \cup [p-m,p] = [-p,p]$, so 
\begin{equation} \label{eq:case2diff}
\big|A - A\big| \ = \ 2p+1.
\end{equation}
Since $p > m+1$, we have $2p+1 > m+p+2$, so $A$ is difference-dominated with $p-m-1$ more differences than sums.  
\end{enumerate}
Thus, $A = [0, m] \cup \{p\}$ is difference-dominated for all $p > m+1$.
\end{proof}

With Theorem \ref{thm:2.4MDTSinterval}, we can construct an MDTS superset $A_{2}$ from any MSTD set $A_{1}$ by letting $m = \max(A_{1})$ and selecting $p > m+1$. 

Next, we construct a MSTD superset $A_{3}$ of the MDTS set $A_{2}$. Let $n = p+5$ if $p$ is even, and let $n = p+2$ if $p$ is odd. In either case, $n$ is odd. Define
\begin{equation} \label{eq:B3def}
B_3 \ := \ [0, n - 1] \setminus \{r\},
\end{equation}
where $2 \le r \le n - 3$ and $A_2 \subset B_3$. Such an $r$ exists since taking $r = n - 3$ gives $r \notin [0, m] \cup \{p\} = A_2$ when $p > m+1$.

By Lemma \ref{lemma1}, $B_3$ satisfies 
\begin{align}
\notag B_3 + B_3 \ &= \ [0, 2n - 2]\\
B_3 - B_3 \ &= \ [-(n-1),n-1].
\end{align}

Let $L^* = \{r\}$ be a 0-dimensional arithmetic progression. Clearly, $\min(L^*) = r$ and $\min(L^*) - 1 \in B_3$. If $n \in L^* + L^*$, then $n = 2r$. As $n$ is odd, $n \ne 2r$, and $n \notin L^* + L^*$.

Now, let $k \ge 2$ and define
\begin{align}
\notag L \ &\coloneqq \ \{n - r\} + n * [1, k] \ = \ \{2n - r, 3n - r, \ldots, (k+1)n - r\} \\
\notag a^* \ &\coloneqq \ 2n - r + (k + 1)n - r \ = \ (k + 3)n - 2r \\
A_3^* \ &\coloneqq \ B_3 \cup L \cup (\{a^*\} -B_3).
\end{align}
Then $A_3 = A_3^* \cup \{n\}$ is MSTD by Theorem \ref{Thm1}. Furthermore, since $A_2 \subset B_3  \subset A_3$, we have $A_2 \subset A_3$.

Now, let $m_2 = a^*$. Clearly, $A_3 \subset [0, m_2]$. We can now repeat the process indefinitely, substituting $m_2$ in for $m$ to generate $A_4$, and so on. This gives a sequence $A_1 \subset A_2 \subset\cdots$ such that $A_{2i-1}$ is sum-dominated and $A_{2i}$ is difference-dominated for all $i \in \mathbb{N}$.

We now apply Method 1 to the Conway set in order to examine the smallest commonalities of this method. Let $A_1 =  \{0,2,3,4,7,11,12,14\}$; $A_1$ has 8 elements and a diameter of 14. 

The two smallest possible $p$ we can pick to generate $A_2$ are $p = 16$ and $p = 17$. However, when we generate $A_3$, we want $n$ to be as small as possible to minimize the cardinality and diameter of $A_3$. If $p = 16$, then $n = 21$, but if $p = 17$, then $n = 19$. Thus, we choose $p = 17$.

For $p = 17$, $A_2 = [0,14] \cup \{17\}$, which has 16 elements and a diameter of 17. 

To generate $A_3$, we set $n = p+2 = 19$. We choose $r$ to be as large as possible to minimize the diameter and cardinality of $L$ and $\{a^*\} - B_3$. For the same reasons, we choose $k$ to be as small as possible. We set $r = 16$ and $k = 2$. Then
\begin{align}
\notag B_3 \ &= \ [0,18]\setminus\{16\} \\
\notag L^* \ &= \ \{16\} \\
\notag L \ &= \ \{22,41\} \\ 
\notag a^* \ &= \ 63 \ \\
\{a^*\} - B_3 \ &= \ [45,63] \setminus \{47\}.
\end{align}

Thus, $A_3 = B_3 \ \cup L \ \cup \ (\{a^*\} -B_3) \ \cup \ \{19\}$ is MSTD. Table \ref{ta:method1} below displays the minimal characteristics for the first few sets in this chain generated by the Conway set. The density of a set refers to that set's cardinality divided by its diameter. The D($A_i$)/D($A_{i-1}$) column measures how much larger a set's diameter is compared to the previous set in the sequence. Note that all entries in the density and growth rate columns are rounded to 3 decimal places, and this convention will be followed throughout the rest of the paper. 


\begin{table}[h]  
\centering 
\caption{Filling in Method 1 Example Sequence} 
\label{ta:method1}
\begin{tabular}{|c|c|c|c|c|c|c|c|c|}
\hline
Set & $\big|A_i+A_i\big|$ & $\big|A_i-A_i\big|$ & Cardinality & Diameter & $\big|A_i\big|/\big|A_{i-1}\big|$ & D($A_i$)/D($A_{i-1}$)& Density \\
\hline
$A_1$ & 26 & 25  & 8 & 14 & N/A & N/A & 0.571\\
$A_2$ & 33 & 35  & 16 & 17 & 2.000 & 1.214 &0.941 \\
$A_3$ & 126 & 125  & 39 & 63 & 2.438 & 3.706 & 0.619\\
$A_4$ & 130 & 131  & 65 & 65 & 1.667 & 1.032 & 1.000\\
$A_5$ & 414 & 413 & 135 & 207 & 2.077 & 3.185 & 0.652\\
$A_6$ & 418 & 419  & 209 & 209 & 1.548 & 1.010 & 1.000\\
$A_7$ & 1278 & 1277 & 423 & 639 & 2.024 & 3.057 & 0.662\\
\vdots & \vdots & \vdots & \vdots & \vdots & \vdots & \vdots & \vdots \\
\hline
\end{tabular}\\
\raggedleft Limiting MSTD density: $0.667$\hspace{0.8cm}
\end{table}

Table \ref{ta:method1} shows that both the cardinality and diameter at least triples between consecutive MSTD sets in this sequence. The exponential growth between the sets generated by this method raises the question of whether a more efficient approach exists for producing the desired sequence.

\begin{rem}
The process for generating such a sequence could be generalized to any MSTD, not just those that start out as subsets of nonnegative integers, due to dilations and translations. 
\end{rem} 

\subsection{Filling in Method 2}

We now construct a second method that incorporates filling in, but is more efficient than Filling in Method 1. This method builds on Theorem 1.1 of Miller, Scheinerman, and Orosz \cite{miller2008}, which requires the following definition.

\begin{defn} \label{def:Pn}
    A set $A \subset [a,b]$ is said to be $P_n$ if both $A+A$ and $A-A$ contain all possible sums or differences with possible exceptions for the first and last $n$ elements, i.e., $[2a+n,2b-n]\subset A+A$ and $[-(b-a)+n,(b-a)-n]\subset A-A$.
\end{defn}

\noindent With this definition in place, we are able to state the following theorem from \cite{miller2008}.

\begin{thm} \label{miller2008thm1.1}
Let $A = L\cup R$ be a $P_n$, MSTD set where $L\subset[1,n],R\subset[n+1,2n]$, and $1,2n\in A$. Fix a $k \geq n$ and let $m$ be arbitrary. Let $M$ be any subset of $[n+k+1, \ n+k+m]$ with the property that it does not have a run of more than $k$ missing elements. Assume further that $n+k+1\notin M$ and set $A(M;k) = L\cup O_1\cup M\cup O_2\cup R'$, where $O_1 = [n+1,n+k]$, $O_2 = [n+k+m+1,n+2k+m]$, and $R' = R+2k+m$. Then $A(M;k)$ is a MSTD set.
\end{thm}

To begin, let $A = A_1\subset[1,2n]$ be a MSTD set which includes both $1$ and $2n$, and which does not include $n$. Suppose further that $A$ is $P_n$. Write $A = L\cup R$ where $L\subset[1,n]$ and $R\subset[n+1,2n]$.

Let
\begin{equation}
    A_2 \ = \ ([0,2n]\setminus\{n\})\cup\{3n\}
\end{equation}
which is difference-dominated (similar calculation in Case 1 of Filling in Method 1), and let
\begin{equation}
    A_3 \ = \ \left(L-n-1)\cup([0,2n]\setminus\{n\}\right)\cup(R+n).
\end{equation}

Analyzing $A_3$ is identical to analyzing $A_3+n+1$, which can be expressed as 
\begin{align}
A_3+n+1 \ = \ L \cup ([n+1,3n+1] \setminus \{2n+1\}) \cup(R+2n+1).  
\end{align} 

Using Theorem \ref{miller2008thm1.1} with the parameters $k = n$, $m = 1$, and $M = \varnothing\subset[2n+1,2n+1]$, we see that $A_3+n+1$ is a MSTD set. Thus $A_3$ is MSTD as well, and we have
\begin{equation}
    A_1\subset A_2\subset A_3.
\end{equation}

For an integer $l \ge 2$, let
\begin{equation}
    A_{2l} \ \coloneqq \ ([(1-l)n,(l+1)n]\setminus\{n\})\cup\{(l+2)n\}
\end{equation}
which is difference-dominated, and let
\begin{equation}
    A_{2l+1} \ \coloneqq \ (L-ln-1)\cup([(1-l)n,(l+1)n]\setminus\{n\})\cup(R+ln).
\end{equation}

Using $k = ln$, $m = 1$, and $M = \varnothing$, we can apply Theorem \ref{miller2008thm1.1} to $$A_{2l+1}+ln+1 = L\cup([n+1,(2l+1)n+1]\setminus\{(l+1)n+1\})\cup(R+2ln+1)$$ to see that $A_{2l+1}$ is sum-dominated. Finally, we have
\begin{align} 
    \notag A_{2l+2} \ &= \ ([\min A_{2l+1},\max A_{2l+1}]\setminus\{n\})\cup\{\max A_{2l+1}+n\}\\
    \ &= \ ([-ln,(l+2)n]\setminus\{n\})\cup\{(l+3)n\} \ = \ A_{2(l+1)},
\end{align}
and
\begin{equation}
    A_{2l} \subset A_{2l+1} \subset A_{2l+2}.
\end{equation}

We summarize this result with the following theorem.
\begin{thm} \label{THMwarmup1method2}
    Let $A = L\cup R$ where $L\subset[1,n]$ and $R\subset[n+1,2n]$, with $1,2n\in A$ and $n\notin A$. Suppose that $A$ is $P_n$ and MSTD. Let $A_1 = A$ and for $l \ge 1$, let $$A_{2l} = ([(1-l)n,(l+1)n]\setminus\{n\})\cup\{(l+2)n\}$$ and let $$A_{2l+1} = (L-ln-1)\cup([(1-l)n,(l+1)n]\setminus\{n\})\cup(R+ln).$$ Then the sequence of sets $A_1\subset A_2\subset \cdots$ alternates between being MSTD and MDTS. Furthermore, for $m \ge 3$ the diameter of the set $A_m$ is n larger than that of the set $A_{m-1}$.
\end{thm}

An example of a set which can be used for this construction is \begin{align}
A_1 \ = \ \{1, 3, 4, 8, 9, 12, 13, 15, 18, 19, 20\},
\end{align}
which is a $P_{10}$ MSTD set with $1,20 \in A_1$ and $10 \notin A_1$. Breaking down $A_1$ into $L$ and $R$, we have \begin{align}
   \notag L \ &= \ \{1,3,4,8,9\} \\
    R \ &= \ \{12,13,15,18,19,20\}.
\end{align}
When applying Method 2 to $A_1$, we get \begin{align}    
\notag A_2 \ &= \ [0,20] \setminus \{10\} \cup \{30\} \\
A_3 \ &= \ \{-10,-8,-7,-3,-2\} \cup  [0,20] \setminus \{10\} \cup \{22,23,25,28,29,30\}.
\end{align}

The cardinalities and diameters of the sets in this sequence are given in Table \ref{FIM2 Table} below.


\begin{table}[h] 
\centering
\caption{Filling in Method 2 Example Sequence} 
\label{FIM2 Table}
\label{tab:method1} 
\begin{tabular}{|c|c|c|c|c|c|c|c|}
\hline
Set & $\big|A_i+A_i\big|$ & $\big|A_i-A_i\big|$ & Cardinality & Diameter & $\big|A_i\big|/\big|A_{i-1}\big|$ & D($A_i$)/D($A_{i-1}$)& Density \\
\hline
$A_1$ & 38 & 37 & 11 & 16 & N/A & N/A  & 0.688\\
$A_2$  & 52 & 61 & 21 & 30 & 1.909 & 1.875 & 0.700\\
$A_3$  & 80 & 79 & 31 & 40 & 1.476 & 1.333 &  0.775\\
$A_4$ & 92 & 101 & 41 & 50 & 1.323 & 1.25 & 0.820\\
$A_5$ & 120 & 119 & 51 & 60 & 1.244 & 1.2  & 0.850\\
$A_6$  & 132 & 141 & 61 & 70 & 1.196 & 1.167 & 0.871 \\
$A_7$  & 160 & 159 & 71 & 80 & 1.164 & 1.143 &0.888\\
\vdots & \vdots &\vdots &\vdots & \vdots & \vdots & \vdots & \vdots \\
\hline
\end{tabular}\\
\raggedleft Limiting MSTD density: $1.000$\hspace{0.8cm}
\end{table}

As seen in Table \ref{FIM2 Table}, both the cardinality and diameter increases by 10 between $A_i$ and $A_{i+1}$ in this sequence for $i \geq 2$. So, while Filling in Method 1 produces the desired sequence for any MSTD set $A_1$, Filling in Method 2 exhibits a more efficient and controlled growth rate between the $A_i$.

\section{Non-Filling in Method}
We now place an additional constraint on the alternating sequence we seek to generate. For any $A_k$, let $a = \min A_k$ and $b = \max A_k$. Then, if an $x \in [a,b]$ is not in $A_k$, we require that $x \notin A_n$ for all $n > k$. In other words, this restriction prohibits filling in when generating the sequence. While filling in provides valid constructions, it can substantially alter or even erase the structure of earlier MSTD sets in the sequence, leading to relatively trivial solutions. With non-filling in methods, each step preserves the properties of the previous sets, leading to more structurally meaningful sequences. We provide a single such method below and several more in a future paper.


\begin{equation}
    A^{\ge x} \ \coloneqq \ \{a\in A:a\ge x\}.
\end{equation}

Let
\begin{equation}
    A_0 \ = \ \{0,1,2,5,8,9,10\}.
\end{equation}

For $l\ge1$, let
\begin{equation}
    A_{2l-1} \ = \ A_0 \ \cup \ (8\cdot[1,l]+\{6,7,9,10\}).
\end{equation}

\noindent We prove that $A_{2l-1}$ is MSTD.\\

First, we show
\begin{equation}
    A_{2l-1}+A_{2l-1} \ = \ [0,16l+20]\setminus\{21\}. \label{Sec3M2Ref1}
\end{equation}
Note that
\begin{equation}
    [0,28]\setminus\{21\}\ \subset\ A_1+A_1\ \subseteq\ A_{2l-1}+A_{2l-1}.
\end{equation}
To construct sums for $[29,16l+20]$, we take $k\in[2,2l]$, and $m,n \in \mathbb{N}$ such that $1 \leq m,n \leq l$ and $m+n = k$. We then have the following cases:
\begin{itemize}
    \item $8k+13 \ = \ (8m+7)+(8n+6)$
    \item $8k+14 \ = \ (8m+7)+(8n+7)$
    \item $8k+15 \ = \ (8m+9)+(8n+6)$
    \item $8k+16 \ = \ (8m+9)+(8n+7)$
    \item $8k+17 \ = \ (8m+10)+(8n+7)$
    \item $8k+18 \ = \ (8m+9)+(8n+9)$
    \item $8k+19 \ = \ (8m+10)+(8n+9)$
    \item $8k+20 \ = \ (8m+10)+(8n+10)$.
\end{itemize}

Thus, we have proven \eqref{Sec3M2Ref1}. Next, we show
\begin{equation}
    A_{2l-1}-A_{2l-1} \ = \ [-8l-10,8l+10]\setminus\{-8l-3,8l+3\}. \label{S3M2R2}
\end{equation} Note that
\begin{equation}
    [-10,10]\ \subset\ A_1-A_1\ \subseteq\ A_{2l-1}-A_{2l-1}.
\end{equation}
For $k\in[1,l-1]$, we have 
\begin{equation*}
    8k+3 \ = \ (8(k+1)+9)-14,
\end{equation*}
and for $k\in[1,l]$ we have the following cases:
\begin{itemize}
    \item $8k+4 \ = \ (8k+6)-2$
    \item $8k+5 \ = \ (8k+6)-1$
    \item $8k+6 \ = \ (8k+6)-0$
    \item $8k+7 \ = \ (8k+7)-0$
    \item $8k+8 \ = \ (8k+9)-1$
    \item $8k+9 \ = \ (8k+9)-0$
    \item $8k+10 \ = \ (8k+10)-0$.
\end{itemize}

Lastly, if we assume without loss of generality that $m > n$, we note that there exists no combination of $m \in 8l+\{6,7,9,10\}$ and $n \in A_0$ such that $m-n = 8l+3$. This is enough to conclude $\pm(8l+3) \notin A_{2l-1}-A_{2l-1}$ since if $m \notin 8l+\{6,7,9,10\}$ or $n \notin A_0$, then $m-n < 8l+3$.

Thus, we have proven \eqref{S3M2R2}, and we have that $A_{2l-1}$ is MSTD with 
\begin{equation}
    16l+20 \ = \ \big|A_{2l-1}+A_{2l-1}\big| \ > \ \big|A_{2l-1}-A_{2l-1}\big| \ = \ 16l+19. \label{16l+20>16l+19}
\end{equation}
Furthermore, we observe that 
\begin{equation}
    A_1\subset A_3\subset\cdots \subset A_{2l-1}\subset \cdots.
\end{equation}

Now, let
\begin{equation}
    A_{2l} \ = \ A_{2l-1}\cup\{8l+14\}.
\end{equation}

\noindent We show that $A_{2l}$ is MDTS.\\

\indent Since $a+8l+14 \neq 21$ for any $a\in A_{2l-1}$, the new sums are $2\cdot(8l+14)=16l+28$ and the elements of the set
\begin{equation}
    \{a+8l+14:a\in A_{2l-1}^{\ge8l+7}\} \ = \ \{16l+21,16l+23,16l+24\}.
\end{equation}

Note that we exclude $8l+6$ from the construction of the new sums because $8l+6+8l+14 = 8l+10+8l+10 \in A_{2l-1}+A_{2l-1}$. Next, since $\pm(8l+3)\notin (A_{2l}-A_{2l})$, the new differences are the numbers
\begin{equation}
    \pm(\{8l+14\}-\{0,1,2\}) \ = \ \pm\{8l+12,8l+13,8l+14\}.
\end{equation}
Therefore, there are 4 new sums and 6 new differences. Combining this with \eqref{16l+20>16l+19}, we get \begin{align}
    16l+24 \ = \ \big|A_{2l}+A_{2l}\big| \ < \ \big|A_{2l}-A_{2l}\big| \ = \ 16l+25
\end{align}
and so $A_{2l}$ is MDTS. Additionally, we have
\begin{equation}
    A_{2l-1}\subset A_{2l}\subset A_{2l+1}
\end{equation}
and thus the sequence $A_1\subset A_2 \subset \cdots$ alternates between being MSTD and MDTS. Furthermore, filling in is not used in this method since every added element is greater than the maximum of the previous set. The cardinalities and diameters of the first few sets in this sequence are given in Table \ref{NFM2Table} below.

\begin{table}[h]
\centering
\caption{Non-Filling in Method 2 Example Sequence} 
\label{NFM2Table}
\begin{tabular}{|c|c|c|c|c|c|c|c|}
\hline
Set & $\big|A_i+A_i\big|$ & $\big|A_i-A_i\big|$ & Cardinality & Diameter & $\big|A_i\big|/\big|A_{i-1}\big|$ & D($A_i$)/D($A_{i-1}$)& Density \\
\hline
$A_1$ & 36& 35 & 11 & 18& N/A & N/A & 0.611 \\
$A_2$ & 40&41 & 12 & 22& 1.091 & 1.222 & 0.545 \\
$A_3$ & 52&51 & 15 & 26& 1.25 & 1.182 & 0.577 \\
$A_4$ & 56&57 & 16 & 30& 1.067 & 1.154 & 0.533 \\
$A_5$ & 68&67 & 19 & 34& 1.188 & 1.133 & 0.559 \\
$A_6$ & 72&73 & 20 & 38 & 1.053 & 1.118 & 0.526\\
$A_7$ & 84&83 & 23 & 42 & 1.15 & 1.105 &0.548\\
\vdots & \vdots &\vdots &\vdots & \vdots &  \vdots & \vdots & \vdots \\
\hline
\end{tabular}\\
\raggedleft Limiting MSTD density: $0.500$\hspace{0.8cm}
\end{table}

Each set in this sequence has a diameter exactly 4 larger than the previous set in the sequence. Furthermore, the diameter increases by 8 and the cardinality by 4 between consecutive MSTD sets in the sequence.

This sequence illustrates the idea behind the following theorem.
\begin{thm} \label{NFIMethod2Theorem}
    Suppose there are sets $L,R\subset[0,n]$ such that the following are true.
    \begin{enumerate}
        \item $n\in L,R$
        \item $[0,n-1]\subset (L+L)$
        \item $[0,n-1]\subset (R+R)$
        \item $[0,n-1]\not\subset[R+L]$.
    \end{enumerate}
    (Note that the second and third conditions imply that $0\in L,R$). Choose a suitable $m \ge n$ such that elements of $[n,m]$ fill out the sums in $[n+1,2m+n-1]$ (cf. \cite{Zhao2010}), lacking up to one element. Then, set \begin{align}
        A_1 = L\cup[n,m]\cup(m+n-R),
    \end{align}
    and for $l\ge1$, set
    \begin{align}
        A_{2k+1} \ = \ A_{2k-1} \ \cup \ (m+(k+1)n-R).
    \end{align}
    Then $A_{2k+1}$ is MSTD, and there may exist $A_{2k}$ which is MDTS such that $A_{2k-1}\subset A_{2k}\subset A_{2k+1}$.
\end{thm}
\begin{proof}
Since
\begin{align}
    \notag &[0,n] \subset (L+L) \\
    &2m+n \ = \ m+(m+n-0)
\end{align}
and
\begin{align}
    \notag[2m+n+1,2m+2n] \ &\subset \ (m+n-R)+(m+n-R)\\
    &=\ 2m+2n-(R+R),
\end{align}
we have $\big|A_1+A_1\big| = 2m+2n+1$ or $\big|A_1+A_1\big| = 2m+2n$.

The maximum value of $\big|A_1-A_1\big|$ is $2m+2n+1$, so in order to show $A_1$ is MSTD we must show that $A_1-A_1$ is missing at least one pair of elements. This is true because elements of $(A_1-A_1)^{\ge m+1}$ must be of the form $(m+n-r)-l$ with $r\in R,l\in L$. From our fourth condition, we know that $r+l$ cannot take all values in $[0,n-1]$, and so $m+n-r-l$ cannot take all values in $[m+1,m+n]$. This causes the desired gap in the difference set, so $A_1$ is MSTD.

We then set \begin{align}
    A_3 \ = \ A_1 \ \cup \ (m+2n-R).
\end{align} 
We have \begin{align}
  \notag [2m+2n+1,2m+3n] \ &\subset \ (m+2n-R)+(m+n-R) \\
  [2m+3n+1,2m+4n] \ &\subset \ (m+2n-R)+(m+2n-R),
\end{align} 
which implies that
\begin{equation}
    [2m+2n+1,2m+4n] \ \subset \ A_3+A_3.
\end{equation}
Furthermore, $A_3-A_3$ is now missing at least one difference in $[m+n+1,m+2n]$ due to the fourth condition, so $\big|A_3+A_3\big|\ge2m+4n$ and $\big|A_3-A_3\big| \le 2m+4n-1$. Thus $A_3$ is MSTD with $A_1 \subset A_3$.

Continuing the sequence, appending $m+(k+1)n-R$ gives new possible sums in the range $[2m+2kn+1,2m+2(k+1)n]$ to consider. The sums are all realized since
\begin{equation}
    [2m+2kn+1,2m+(2k+1)n] \subset (m+kn-R)+(m+(k+1)n-R)
\end{equation}
and
\begin{equation}
    [2m+(2k+1)n+1,2m+2(k+1)n] \subset (m+(k+1)n-R)+(m+(k+1)n-R).
\end{equation}
Then, since $A_{2k-1}-A_{2k-1}$ always lacks at least two differences (again one in $(A-A)^{\ge m+kn+1}$, and its negative), $A_{2k-1}$ is MSTD. The diameter of a MSTD set in the sequence is $n$ larger than the previous.
\end{proof}

\begin{rem}
    We may also generalize the conditions as follows. We require only that the number of elements in $[0,n-1]$ missing from $L+L$ is less than twice the number of elements in $[0,n-1]$ missing from $L+R$. The slowest-growing sequence obtained with the new conditions uses $n = 7$, $L = \{0,1,3,7\}$, $R = \{0,1,2,4,7\}$, and $m = 8$.
\end{rem}

\newpage 
\section{Growth Rates}
We conclude by giving a table that compares the growth characteristics of the various methods we covered. Note that the $\big|A_1\big|$ and $A_1$ Diam. columns give the smallest possible cardinality and diameter for a set $A_1$ which a method can be applied to. Additionally, the Card. Rate and Diam. Rate columns measure the growth rate between consecutive MSTD sets in the sequence generated by the minimal $A_1$ for that method.

\begin{table}[h]
\centering
\textbf{\caption{Minimal Method Growth Characteristics}} \vspace{0.1cm}
\label{tab:method1} 
\begin{tabular}{|c|c|c|c|c|c|}
\hline
\textbf{Method} & $\mathbf{\big|A_1\big|}$ & $\mathbf{A_1}$ \textbf{Diam.} & \textbf{Card. Rate} & \textbf{Diam. Rate} & \textbf{Type} \\
\hline
Filling in 1 & 8 & 14 & $>3\cdot\big|A_{2i-1}\big|$ & $>3\cdot$ Diam$(A_{2i-1})$ & Exponential\\
\hline 
Filling in 2 & 11 & 19 & 20 & 20 & Linear\\ 
\hline 
Non-Filling in & 11 & 18 & 4 & 8 & Linear\\
\hline
\end{tabular}
\end{table}

\vspace{0.15cm}
\printbibliography

\vspace{0.8cm}

\end{document}